\documentclass[a4paper,12pt]{article}

\usepackage{amssymb, amsmath}
 \usepackage{amsthm}

\title{On non-separated zero sequences  of solutions of a linear differential equation}
 \author{Igor Chyzhykov, Jianren Long}

\def\Z{\Bbb Z}
\def\N{\Bbb N}

\newcommand{\D}{\mathbb{D}}
\newcommand{\Bl}{\Bigl(}
\newcommand{\Br}{\Bigr)}
\newcommand{\Bm}{\Bigl|}

\newtheorem{thm}{Theorem}
\newtheorem{rmk}{Remark}

\newtheorem{cor}[thm]{Corollary}

\newtheorem{atheorem}{Theorem}

\newtheorem{alemma}[atheorem]{Lemma}

\newtheorem{acor}[atheorem]{Corollary}

\begin{document}
\maketitle
\begin{abstract}
Let $(z_k)$ be a  sequence of distinct points in the unit disc  $\mathbb{D}$ without limit points there.
We are looking for  a function $a(z)$ analytic in $\mathbb{D}$ and  such that
    $f''+a(z)f=0$
 possesses a solution having zeros precisely at the points $z_k$, and the resulting function  $a(z)$ has `minimal' growth.
  We focus on the case of non-separated sequences $(z_k)$  in terms of the pseudohyperbolic distance when the coefficient $a(z)$ is of zero order, but $\sup_{z\in \mathbb{D}} (1-|z|)^p |a(z)|=+\infty$ for any $p>0$.
We established a  new estimate  for the maximum modulus of $a(z)$ in terms of the functions $n_z(t)=\sum_{|z_k-z|\le t} 1 $ and
$N_z(r)=\int_0^r
\frac{(n_z(t)-1)^+}{t}dt.$ The estimate is sharp in some sense.
The main result relies on a new interpolation theorem.


MathSubjClass: 34C10, 30C15,  30H99, 30J99.
\end{abstract}

\section{Introduction}

Let  $(z_n)$ be a sequence of different complex numbers in the unit disc $\mathbb{D}=\{ z: |z|<1\}$, and let $U(z,t)=\{\zeta\in \mathbb{C} : |\zeta-z|<t\}$.  In the sequel, the symbol $C$ stands for positive constants which depend on the parameters indicated, not necessarily the same at each occurrence.

The aim of the paper is twofold.
On one hand, we are interested in zeros of solutions
 of
\begin{equation}\label{e:oscil_eq}
    f''+a(z)f=0,
\end{equation}
where $a(z)$ is an analytic function  in $\mathbb{D}$. On the other hand it leads us to some interpolation problems for corresponding classes of analytic functions in $\mathbb{D}$.

\subsection{Oscillation of solution}
It was proved by \v Seda \cite{Seda} 
 that given a sequence of distinct complex numbers $(z_n)$ with no finite limit points there exists an entire function  $a(z)$ such that the equation (1) has an entire solution $f$ with the zero sequence~$(z_n)$. This result was recently generalized for an arbitrary domain $G\subset \mathbb{C}$ when the condition $f(z_n)=0$ is replaced with
 \begin{equation}\label{e:inter_prob}
f(z_k)=b_k,
\end{equation}
 $(b_n)$ being an arbitrary sequence \cite{Vyn_Shav12}.


We are interested in description of zero  sequences  $(z_n)$ of solutions of (1) where $a(z)$ belongs to some growth class.
 Let  $ \mathcal{A}^{-p}$ be the Banach space of analytic functions in $\D$ with the norm
$$\|a \|_{\mathcal{A}^{-p}}:=\sup_{|z|<1} (1-|z|^2)^p|a(z)|.$$


In \cite{GrNikRat} the authors described behavior of the coefficient when all zero-free solutions belong to some space.
We restrict ourselves to the case when $f$ has infinitely many zeros is in the focus of the paper.
%
%

Let $\varphi(z, w)=\frac{z-w}{1-\bar z w}$.  Let $\sigma(z, w)=|\varphi_z(w)|$  denote the pseudohyperbolic distance in $\mathbb{D}$.
A sequence $(z_n)$ in the unit disc is called {\it uniformly separated} if $$\inf_{j} \prod_{n\ne j} \sigma(z_n, z_j)>0.$$

A positive Borel measure $\mu$ is  a Carleson measure if and only if there exists a constant $K$ such that $\mu (Q_\delta)\le K \delta$ for any   Carleson box $$Q_\delta=\{ \zeta \in \overline{\mathbb{D}}: |\zeta|\ge 1-\delta, |\arg \zeta -\varphi |
\le \pi \delta\}.$$

The following result describes coefficients $a(z)$ such that zero sequence of a solution \eqref{e:oscil_eq} is uniformly separated.
\begin{atheorem} 
 {  A sequence $Z$ in the unit disc is  the zero-sequence of a non-trivial solution of (1) such that $|a(z)|^2 (1-|z|^2)^3 dm(z)$ is a Carleson measure if and only if $Z$ is uniformly separated.} 
\end{atheorem}
The necessity is proved by Gr\"ohn, Nicolau, R\"atty\"a in \cite{GrNikRat}. The sufficiency is established  recently by Gr\"ohn in \cite{Gron19}.

The following problem was formulated in \cite{Heit_cmft}.

{\bf Problem.}  {\sl Let $(z_k)$ be a sequence of distinct points in $\D$ without limit points there. Find a function $a(z)$, analytic in $\D$ such that \eqref{e:oscil_eq} possesses a solution having zeros precisely at the points $z_k$. Estimate the growth of the resulting function  $a(z)$.}

The next result is closely connected to the  problem. In order  to formulate   it we need more notation.
A sequence $Z=(z_n)$ in the unit disc is called separated if $\inf_{n\ne j} \sigma(z_n, z_j)>0$.
The uniform density of a sequence $(z_n)$ (\cite{Seip93}) 
is defined by
 $$   {D}^+(Z)=\limsup_{r\to 1-} \sup_{z \in {\D}} \frac{\sum\limits_{\frac 12< \sigma(z, z_j)<r}\log \frac 1{\sigma(z, z_j)}}{\log \frac 1{1-r}}.$$



 \begin{atheorem}[{\cite[Theorem 1]{Gron19}}] \label{t:gr19} If $Z=(z_k)$ is a separated sequence in the unit disc with   $D^+(Z)<1$ then there exists $a\in \mathcal{A}^{-2}$ such that (1) admits a nontrivial solution that vanishes on $Z$.

Conversely, if  $a\in \mathcal{A}^{-2}$ and $f$ is a nontrivial solution of (1) whose zero-sequence is $Z$, then $Z$  is separated  and contains at most one point if $\|a\|_{\mathcal{A}^{-2}}\le 1$, while
$D^+(Z)\le (2\pi+1) \frac{C}{(1-C)^2}, $ where $C=\sqrt{1- \frac{2\sqrt{\|a\|_{\mathcal{A}^{-2}}}}{\|a\|_{\mathcal{A}^{-2}}+1}}$.
\end{atheorem}

\begin{rmk} The proof of Theorem \ref{t:gr19} uses essentially the  interpolation result by K. Seip.
In the first part of Theorem \ref{t:gr19}, $Z$ is a {\bf subset} of the zero set of $f$.
\end{rmk}

For an analytic function $f$ in  $\mathbb{D}$ we denote $M(r,f)=\max
\{|f(z)|: |z|=r\}$, $r\in (0,1)$. Let
$n_\zeta(t)=\sum_{|z_k-\zeta|\le t} 1 $ be the number of the members of the sequence  $(z_k)$ satisfying $|z_k-\zeta|\le t$.
We write  $$N_\zeta(r)=\int_0^r
\frac{(n_\zeta(t)-1)^+}{t}dt.$$

Let  $\psi\colon [1,+\infty)\to \mathbb{R}_+$ be a nondecreasing function. We define
$$ \tilde \psi(x)=\int_1^x \frac {\psi(t)}t dt.$$

\begin{rmk} \rm In the case $\psi(x)=\log  ^p x$, $p\ge 0$ we get $\tilde \psi(x)=\frac 1{p+1} \log  ^{p+1} x$, and in
the case $\psi(x)=x^\rho $, $\rho>0$ we have $\tilde \psi(x)=\frac 1{\rho} (x^{\rho} -1)$.
\end{rmk}

Let, in addition, $\psi$ have  finite order in the sense of P\'olya (\cite{DrSh}), i.e.
\begin{equation}\label{e:polya}
\psi(2x)=O(\psi(x)), \quad x\to+\infty.
\end{equation}
%

\begin{atheorem}[{\cite{ChyShep}}] \label{t:zeros_de_old}
Let  $(z_n)$ be a sequence of distinct complex numbers in $\mathbb{D}$.
Assume that for some nondecreasing unbounded function  $\psi\colon [1,+\infty)\to \mathbb{R}_+ $ satisfying \eqref{e:polya} we have \begin{equation}\label{e:concentr_cond_psi}
 \exists C>0:\    \ \forall n\in \mathbb{N}   \quad N_{z_n} \Bigl(\frac{1-|z_n|}2 \Bigr)\le C  \psi\Bl \frac 1{1-|z_n|}\Br.
 \end{equation}
 Then there exists an analytic function $a$ in $\D$ satisfying
\begin{equation*}
 \exists C>0: \log  M(r,a)\le C \tilde \psi \Bigl( \frac 1{1-r} \Bigr), \quad r\in (0,1)
\end{equation*}
such that \eqref{e:oscil_eq} possesses a solution $f$  having zeros  precisely  at the points $z_k$, $k\in \mathbb{N}$.
\end{atheorem}

\begin{acor}
 If for some $\rho>0$ a sequence $(z_k)$ satisfies the condition
\begin{equation*}
 \exists C>0: N_{z_k} \Bigl(  \frac{1-|z_k|}{2}\Bigr) \le C \Bigl( \frac{1}{1-|z_k|}\Bigr)^\rho,
\end{equation*}
then there exists a function $a$ analytic in $\D$ satisfying $\log  M(r,a)=O((1-r)^{-\rho})$, $r\in (0,1)$ such that
possesses a solution $f$  having zeros  precisely  at the points $z_k$, $k\in \mathbb{N}$.
\end{acor}

The following theorem is based on  an example due to J. Gr\"ohn and J. Heittokangas \cite{Gr_Heit}. It shows that the statement of the corollary is sharp.
\begin{atheorem} \label{t:zeors_sharp} For arbitrary $\rho>0$ there exists a sequence of distinct numbers $\{z_n\}$ in $\D$ with the following properties:
\begin{itemize}
 \item [i)] $N_{z_k} \Bigl(  \frac{1-|z_k|}{2}\Bigr) \le C \Bigl( \frac{1}{1-|z_k|}\Bigr)^\rho$, $k\in \mathbb{N}$;
\item [ii)] $(z_k)$ cannot be the zero sequence of a solution of \eqref{e:oscil_eq}, where $\log  M(r,a)=O((1-r)^{-\rho+\varepsilon_0})$ for  any $\varepsilon_0>0$.
\end{itemize}
\end{atheorem}

In \cite{Gr_Heit}  the case when the coefficient $a\in \mathcal{A}^{-p}$, $p>2$, is considered.
Some other problems on zeros of solutions are considered in a survey  \cite{Heitt_sur}.

The aforementioned results give  a complete solution to  the Problem in the cases  when $a\in \mathcal{A}^{-2}$ and the order $a$ is finite and positive. In the intermediate cases, when $a$ is of zero order, but outside of $ \mathcal{A}^{-2}$, zero sets of solutions of \eqref{e:oscil_eq} is not described  completely.  The aim of the paper is to fill this gap. In particular, we improve Theorem \ref{t:zeros_de_old} and obtain sharp, in some sense, estimates of $a(z)$ in terms of the zero distribution of a solution of  \eqref{e:oscil_eq}. Our proof relies on a new interpolation result.

%

\subsection{Interpolation in the unit disc}

For the  space
$\mathcal{A}^{-n}$, $n>0$,  an interpolation set is defined by the condition that for every sequence $(b_k)$ with  $(b_k (1-|z_k|)^n) \in l^\infty$ there is a function $f\in \mathcal{A}^{-n}$ satisfying \eqref{e:inter_prob}.
These sets were described by K.~Seip in  \cite{Seip93}. Namely, necessary and sufficient that $(z_k)$  be an interpolation set for $\mathcal{A}^{-n}$ is that $(z_n)$ be separated, i.e. $\inf _{j\ne k} \sigma(z_k,
z_j) >0$, and $\mathcal{D}^+(Z)<n$.

In 2002 A.\ Hartmann  and X.\ Massaneda \cite{HarMas} proved that  condition
\begin{equation*}
 \exists \delta\in (0,1)\  \exists C>0 \ \forall n\in \mathbb{N} :  \quad N_{z_n} (\delta (1-|z_n|)) \le \eta\Bl \frac C
{1-|z_n|}\Br.
 \end{equation*}
is  necessary and sufficient for $Z$ to be an interpolation set for a class of
growth functions $\eta$ containing all power functions. They also describe interpolation sequences in the unit ball in $\mathbb{C}^n$ in the similar situation.
 Note that
the proof of sufficiency in \cite{HarMas} is based on $L^2$-estimate for the solution to a $\bar \partial$-equation and is non-constructive.

The following theorem was established in \cite{ChyShep}. It gives sufficient conditions for interpolation sequences in classes of analytic functions of moderate growth in the unit disc and  complements  Hartmann and Massaneda's result when $\psi(t) $ grows slowly than any power function.

\begin{atheorem} \label{t:sufficient_cond_interpol} Let  $(z_n)$ be a sequence of distinct complex numbers in $\mathbb{D}$.
Assume that for some nondecreasing unbounded function  $\psi\colon [1,+\infty)\to \mathbb{R}_+ $ satisfying \eqref{e:polya} the condition 
\eqref{e:concentr_cond_psi} is valid.
Then for any sequence  $(b_n)$ satisfying
 \begin{equation}\label{e:b_n_con_psi}
 \exists C>0:\    { \log  |b_n|} \le C {\tilde \psi\Bl \frac {1}{1-|z_n|}}\Br , \quad n\in \mathbb{N},
     \end{equation}
there exists an analytic function $f$ in  $\mathbb{D}$ with the property \eqref{e:inter_prob}
and
 \begin{equation}\label{e:growth_con_psi}
 \exists C>0:\ \log  M(r,f)\le  C \tilde\psi \Bl \frac {1}{1-r}\Br. \end{equation}
 \end{atheorem}

The class $\mathcal{R}$ consists of functions $\psi \colon [1,+\infty)\to \mathbb{R}_+$ which are nondecreasing, and  such that
$\tilde \psi(r)=O(\psi(r))$ as $r\to+\infty$. We note that the power function $x^\rho$, $\rho>0$, belongs to $\mathcal{{R}}$.  The previous theorem becomes a criterion if $\psi \in \mathcal{{R}}$ (\cite[Theorem 5]{ChyShep}).

In 2007  A. Borichev, R.\ Dhuez and K.\ Kellay \cite{Borich_k} solved an interpolation problem in classes of functions of arbitrary growth in both the complex plane and the unit disc.  Following \cite{Borich_k} let $h\colon [0,1)\to [0,+\infty)$ such that  $h(0)=0$, $h(r) \uparrow \infty$ $(r \to 1-)$.  Denote by $\mathcal{A}_h$  and $\mathcal{A}_h^p$, $p>0$ the Banach spaces  of analytic functions on $\D$ with the norms
$$ \| f\|_h=\sup_{z\in \D} |f(z)| e^{-h(z)}<+\infty, \quad  \| f\|_{h, p}= \biggl( \int_{\D} |f(z)| ^p e^{-p h(|z|)} dm_2(z)\biggr)^{\frac 1p}, $$
respectively.
We then suppose that  $h\in C^3([0,1))$, $\rho(r):=(\Delta h(r))^{-\frac12}\searrow 0$, and $\rho'(r)\to 0$ as $r\to 1-$, for all
$ K>0$:
 \begin{equation}\label{e:rho_slow_var}
   \rho(r+x)\sim \rho(r)\; \text{  for } |x|\le K\rho(r), \; r\to 1-
   \end{equation}
    provided that $K\rho(r)<1-r$, and either $\rho(r)(1-r)^{-c}$ increases for some finite $c$ or $\rho'(r) \log  \rho(r)\to 0$ as $r\uparrow1$. Note that these assumptions imply $h(r)/\log  \frac 1{1-r} \to+\infty$ $(r\to 1-)$.

Given such an $h$ and a sequence $Z=(z_k)$ in $\D$ denote by $$\mathcal{D}_\rho^+(Z)=\limsup_{R\to\infty} \limsup_{|z|\to1-} \frac{\mathop{card}(Z\cap U(z, R\rho(z)))}{R^2}.$$

\begin{atheorem}[Theorem 2.3 \cite{Borich_k}] \label{t:bdk} A sequence $Z$ is an interpolation  set for $\mathcal{A}_h (\D)$ if and only if $\mathcal{D}_\rho^+(Z)< \frac 12$ and
\begin{equation}\label{e:rho_separ}
\inf\limits _{k\ne n} \dfrac{|z_k-z_n|}{\min \{\rho(|z_k|), \rho(|z_n|)\}}>0.
\end{equation}
 \end{atheorem}
The similar description holds for interpolation sets for the classes $\mathcal{A}_{h}^p (\D)$, $p>0$~(\cite{Borich_k}).
\section{Main results}
In this paper we are mostly interested in the case where $\psi(r) $ is a slowly growing function unbounded with respect to  $\log \frac 1{1-r}$ as $r\to 1-$, in particular, $\psi\not \in \mathcal{{R}}$.  Theorem \ref{t:sufficient_cond_interpol} seems no longer to be sharp for such functions $\psi$.

For  $s=[\rho]+1$, where $\rho=\rho^*[\psi]$, we consider a canonical product of the form
\begin{equation}\label{e:can_prod}
P(z)=P(z, Z,s)=\prod_{n=1}^\infty E\Bigl( \frac{1-|z_n|^2}{1-\bar z_n z},s\Bigr),
\end{equation}
where  $E(w,0)=1-w$,  $E(w,s)=(1-w)\exp \{w+ w^2/2 +\dots+ w^s/s\}$, $ s\in \N$.
This product is an analytic function in $\mathbb{D}$ with the zero sequence  $Z=(z_n)$ provided  $\sum\limits_{z_n\in Z} (1-|z_n|)^{s+1}<\infty$.

\begin{rmk} The  P\'olya order
  $\rho^*[\psi]$    (\cite{DrSh})  characterized by the condition that for any $\rho>\rho^*[\psi]$, we have
\begin{equation}\label{e:psi_con}
 \psi(Cx)\le C^\rho \psi(x), \quad x, C\to \infty.
\end{equation}
\end{rmk}


The following result allows to relax the assumption on $N_{z_n}(t)$ in comparison with Theorem \ref{t:sufficient_cond_interpol}.
\begin{thm} \label{t:suff_interpol2} Let  $(z_n)$ be a sequence of distinct complex numbers in $\mathbb{D}$.
Assume that for some nondecreasing unbounded function  $\psi\colon [1,+\infty)\to \mathbb{R}_+ $ satisfying \eqref{e:polya} we have that \eqref{e:nu_est_rho} and either
\begin{equation}\label{e:concentr_cond_til_psi}
 \exists C>0:\    \ \forall n\in \mathbb{N}   \quad N_{z_n} \Bigl(\frac{1-|z_n|}2 \Bigr)\le C  \tilde \psi\Bl \frac 1{1-|z_n|}\Br,
 \end{equation}
or
\begin{equation}\label{e:ln_prime2}
 \forall n\in \mathbb{N}:\;    -\log  \bigl((1-|z_n|)|P'(z_n)|\bigr) \le  C\tilde\psi\Bl \frac 1{1-|z_n|}\Br,
 \end{equation} or
 \begin{equation}\label{e:ln_min_max}
 \forall n\in \mathbb{N}:\;    -\log  |B_n(z_n)| \le  C\tilde\psi\Bl \frac 1{1-|z_n|}\Br,
 \end{equation}
  holds, where $B_n(z)=P(z)/E(\frac{1-|z_n|^2}{1-\bar z_n z},s)$, $P$ is the canonical product defined by \eqref{e:can_prod}, $s
  \ge [\rho]+1$, where $\rho$ is Pol\'ya order of $\psi$.
Then for any sequence  $(b_n)$ satisfying \eqref{e:b_n_con_psi}
there exists an analytic function $f$ in  $\mathbb{D}$ with the properties \eqref{e:inter_prob}
and \eqref{e:growth_con_psi}.
 \end{thm}

Hypotheses  similar to \eqref{e:ln_prime2} are frequently used in interpolation problems (e.g. \cite{BerTay})
The next theorem addresses the problem formulated in the introduction.

\begin{thm} \label{t:zeros_de}
 Let conditions of Theorem \ref{t:suff_interpol2} be satisfied. Then there exists an analytic function $a$ in $\D$ satisfying
\begin{equation*}
 \exists C>0: \log  M(r,a)\le C \tilde \psi \Bigl( \frac 1{1-r} \Bigr), \quad r\in (0,1)
\end{equation*}
such that \eqref{e:oscil_eq} possesses a solution $f$  having zeros  precisely  at the points $z_k$, $k\in \mathbb{N}$.
\end{thm}

\begin{cor}
 If for some $\rho>0$ and $\beta>0$ a sequence $(z_k)$ satisfies the conditions
\begin{equation*}
 \exists C>0: n_{z_k} \Bigl(  \frac{1-|z_k|}{2}\Bigr) \le C \log^\beta \frac{1}{1-|z_k|},
\end{equation*}
\begin{equation*}
 \exists C>0: N_{z_k} \Bigl(  \frac{1-|z_k|}{2}\Bigr) \le C \log^{\beta+1} \frac{1}{1-|z_k|},
\end{equation*}
then there exists a function $a$ analytic in $\D$ and  satisfying $$\log  M(r,a)=O\Bl \log^{\beta+1} \frac{1}{1-r}\Br,\quad  r\in (0,1)$$ such that
possesses a solution $f$  having zeros  precisely  at the points $z_k$, $k\in \mathbb{N}$.
\end{cor}

This corollary is sharp in the following sense
\begin{thm} \label{t:zeors_sharp_slow} For arbitrary $\eta_1, \eta_2>0$ there exists a sequence of distinct numbers $(z_n)$ in $\D$ with the following properties:
\begin{itemize}
 \item [i)] $\exists C>0$ :  $n_{z_k} \Bigl(  \frac{1-|z_k|}{2}\Bigr) \le C\log^{\eta_1} \frac{1}{1-|z_k|}$, $k\in \mathbb{N}$;
     \item [ii)] $\exists C>0$ :  $N_{z_k} \Bigl(  \frac{1-|z_k|}{2}\Bigr) \le C \log^{1+\eta_1+\eta_2} \frac{1}{1-|z_k|}$, $k\in \mathbb{N}$;
\item [iii)] $(z_k)$ cannot be the  zero sequence of a  solution of \eqref{e:oscil_eq} where $\log  M(r,a)=O(\log^{1+\eta} \frac{1}{1-r})$, $\eta<\eta_2$.
\end{itemize}
\end{thm}


Since Theorem \ref{t:gr19} effectively uses the notion of the uniform density, one may ask whether it is possible to use $\mathcal{D}_\rho^+$ density to solve the Problem.
The next theorem gives an estimate of the growth of $a(z)$ under an assumption in terms of the density introduced by Borichev, Dhuez, and Kellay \cite{Borich_k}.

\begin{thm} \label{t:zeros_de1} Let $h\in C^{2}[0,1)$ be an increasing function with $h(0)=0$ and such that for $\rho(r)=(\Delta h(r))^{-\frac12}$ \eqref{e:rho_slow_var} holds and $\frac{(1-r)h'(r)}{h(r)}$ is bounded. Let the function $\sigma(r)=(1-r)^2/\rho^2(r)\nearrow \infty$ as $r\nearrow  1-$, and  satisfy  $\sigma((1+r)/2) =O(\sigma(r))$, $r\in [1/2,1)$. 
 Suppose that $\mathcal{D}_\rho^+(Z)< \infty$ and \eqref{e:rho_separ} holds.
 Then there exists an analytic function $a$ in $\D$ satisfying
\begin{equation*}
 \exists C>0: \log  M(r,a)\le C  h(r), \quad r\in (0,1)
\end{equation*}
such that \eqref{e:oscil_eq} possesses a solution $f$  having zeros  precisely  at the points $z_k$, $k\in \mathbb{N}$.
\end{thm}
\begin{rmk} Note that the assumption $(1-r)/\rho(r)\to \infty$ as $r\to 1-$ implies that $h(r)/\log  \frac{1}{1-r}\to \infty$ as $r\to 1-$. On the other hand,
$\frac{(1-r)h'(r)}{h(r)}$  provides that $h(r)(1-r)^{q}$ is bounded for some finite $q>0$.
\end{rmk}
\section{Proofs of the results}
\begin{proof}[Proof of Theorem \ref{t:suff_interpol2}]
  We follow the scheme of the proof of Theorem \ref{t:sufficient_cond_interpol} from \cite{ChyShep}.
  It follows from the estimate \eqref{e:nu_est_rho} and \cite[Lemma 9]{ChyShep} that
$$\sum_{n=1}^\infty \Bigl|\frac{1-|z_n|^2}{1-\bar z_n z}\Bigr|^{s+1}\le C(s) \tilde \psi\Bl \frac 1{1-|z|}\Br, \quad z\in \mathbb{D}.$$

The following two lemmas are important in our arguments.
\begin{alemma}[\cite{ChyShep}]  \label{l:index}
For an arbitrary  $\delta\in (0,1)$, any sequence $Z$ in $\D$ satisfying $\sum_{z_k\in Z} (1-|z_k|)^{s+1}< \infty$, $s\in \Z_+$  there exists a positive constant  $C(\delta, s)$
$$| \log  |B_k(z_k)| + N_{z_k} (\delta(1-|z_k|))|\le C(\delta, s)\sum_{n=1}^\infty \Bigl|\frac{1-|z_n|^2}{1-\bar z_n z}\Bigr|^{s+1},\quad k\to+\infty.$$
\end{alemma}

The next proposition compares some conditions frequently used in interpolation problems (cf. \cite{BerTay}).
\begin{alemma}[{\cite[Proposition 11]{ChyShep}}] \label{p:1}
Given  a function $\psi \in \mathcal{R}$ for $$ \exists C>0\ \forall z
\in \mathbb{D}:\  N_z\Bl\frac {1-|z|}2\Br\leq  C\psi\Bl \frac 1{1-|z|}\Br
$$
it is necessary and sufficient that
\begin{equation}\label{e:nu_est_rho}
 \exists C>0\ \forall z
\in \mathbb{D}:     n_{z} \Bl \frac{1-|z|}2\Br \le  C\psi\Bl \frac 1{1-|z|}\Br,
\end{equation}
and
\begin{equation}\label{e:ln_prime}
 \forall n\in \mathbb{N}:\;    |\log  \bigl((1-|z_n|)|P'(z_n)|\bigr)| \le  C\psi\Bl \frac 1{1-|z_n|}\Br,
 \end{equation}
where $P$ is the canonical product defined by (\ref{e:can_prod}), $s=[\rho^*]+1$, where $\rho^*$ is Polya's order of $\psi$.
\end{alemma}
 Lemma \ref{l:index} directly implies that under the assumption \eqref{e:nu_est_rho} the conditions \eqref{e:concentr_cond_til_psi} and \eqref{e:ln_min_max} are equivalent.

Moreover, one has
\begin{equation}
   \frac{P'(z_n)(1-|z_n|^2)}{\bar z_n}= -B_n(z_n) \exp\Bigl\{1+\frac 12+\dots +\frac 1s\Bigr\}, \label{e:p'est}
\end{equation}
which yields equivalence of \eqref{e:ln_prime} and \eqref{e:ln_min_max}.
Therefore, under assumption \eqref{e:nu_est_rho} any of the hypotheses \eqref{e:concentr_cond_til_psi}--\eqref{e:ln_min_max} imply the other two. Hence, we may assume that the conditions \eqref{e:concentr_cond_til_psi}--\eqref{e:ln_min_max} hold.
Taking into account \cite[Lemma 9]{ChyShep}, which gives an upper estimate of a canonical product,
and  Lemma  \ref{l:index} we deduce
\begin{equation}\label{e:error_est_ro}
|\log  |B_k(z_k)|| \le C(\delta,s) \tilde \psi \Bl \frac 1{1-|z_k|}\Br, \quad k\to+\infty.
\end{equation}

Consider the interpolating  function
\begin{equation*}\label{e:inter_func}
    f(z)=\sum_{n=1}^\infty \frac {b_n}{z-z_n} \frac{P(z)}{P'(z_n)}\Bl \frac{1-|z_n|^2}{1-\bar z_n z}\Br^{s_n-1},
\end{equation*}
where is an appropriate increasing sequence of natural numbers  $(s_n)$ (see \cite{ChyShep}). It is not hard to check that \eqref{e:inter_prob} holds.

Simple arguments (\cite[p.328]{ChyShep}) yield
\begin{gather*}\nonumber
\Bm \frac {1-\bar z_n z}{\bar z_n} \frac {P(z)}{z-z_n}\Bm
\le \exp\Bigl\{C \tilde \psi \Bl\frac {1}{1-|z|}\Br\Bigr\}.
\label{e:growth_interpol_est}
\end{gather*}
Therefore, using our assumption on $(b_n)$ and the above estimates we deduce
\begin{gather} \nonumber
    |f(z)|=\biggl| \sum_{n=1}^\infty b_n \frac{P(z)(1-\bar z_n z)}{\bar z_n (z-z_n)} \frac{\bar z_n}{(1-|z_n|^2)P'(z_n)} \Bl \frac{1-|z_n|^2}{1-\bar z_n z}\Br^{s_n}\biggr| \le \\
    \le \sum_{n=1}^\infty \exp\Bigl\{ C\tilde \psi \Bl \frac {1}{1-|z_n|}\Br \Bigr\} \exp\Bigl\{  {C}\tilde \psi \Bl \frac1{1-|z|}\Br\Bigr \} \frac 1{|B_n(z_n)|} \Bl \frac{1-|z_n|^2}{|1-\bar z_n z|}\Br^{s_n} \le \nonumber
        \\ \le \exp\Bigl\{ C \tilde \psi\Bl\frac {1}{1-|z|}\Br \Bigr\} \sum_{n=1}^\infty \exp\Bigl\{ C \tilde \psi\Bl\frac {1}{1-|z_n|}\Br \Bigr\}  \Bl \frac{1-|z_n|^2}{|1-\bar z_n z|}\Br^{s_n}. \label{e:|f|_est}
\end{gather}

The last estimate coincides with inequality (24) from \cite{ChyShep}, and the end of the proof are in lines of that from \cite[pp. 328--329]{ChyShep}.
\end{proof}

\begin{proof}[Proof of Theorem \ref{t:zeors_sharp_slow}]
 Let $\eta_1, \eta_2 >0$ be given. Let $\varepsilon_n=  e^{-(n \log  3)^{1+\eta_2}}$, $m_n=[(n\log  3)^{\eta_1}]$, $n\in \N$.
  Let $(z_{n,k})$ be the sequence defined  by $$z_{n,k}=1-3^{-n}+k\varepsilon_n/m_n, \quad 0\le k \le m_n-1.$$
Then  $$n_{z_{n,k}}(t)=m_n\asymp \log^{\eta_1}\frac 1{1-|z_{n,k}|}, \quad \varepsilon_n\le t\le  \frac{1-|z_{n,k}|}{2}, \; 0\le k\le m_n-1,\; n \in \N,$$
and the assertion i) holds.

Further,
\begin{equation}\label{e:N_z_low}
  \int_{\varepsilon_n}^{\frac{1-|z_{n,k}|}{2}} \frac{(n_{z_{n,k}}(t)-1)^+}{t}\, dt \asymp m_n\log  \frac{1-|z_{n,k}|}{2\varepsilon_n}\asymp (n\log  3)^{\eta_1+1+\eta_2},
\end{equation}
for $0\le k\le m_n-1$, $n \in \N$. On the other hand,  for the same range of $n$ and $k$ we have
\begin{equation}\label{e:N_z_up}
  \int_0^{\varepsilon_n} \frac{(n_{z_{n,k}}(t)-1)^+}{t}\, dt\le \int_0^{\varepsilon_n}
\frac  {2tm_n}{\varepsilon_nt} dt=2m_n\asymp  (n\log  3)^{\eta_1}.
\end{equation}
Combining \eqref{e:N_z_low} and \eqref{e:N_z_up} we deduce
\begin{gather*}
 N_{z_{n,k}} \Bigl( \frac{1-|z_{n,k}|}{2}\Bigr)=\biggl(  \int_0^{\varepsilon_n}+  \int_{\varepsilon_n}^{\frac{1-|z_{n,k}|}{2}} \biggr)  \frac{(n_{z_{n,k}}(t)-1)^+}{t}\, dt\asymp (n\log  3)^{\eta_1+1+\eta_2}\sim \\ \sim \log ^{\eta_1+1+\eta_2} \frac{1}{1-|z_{n,k}|},\quad  0\le k\le m_n-1, n \to \infty.
\end{gather*}

Thus, assertion ii) is proved.

To prove assertion iii) we assume on the contrary that there exists a solution $f=Be^g$ of \eqref{e:oscil_eq} having the zero sequence $(z_n)$, where $B$ is the Blaschke product, and such that
\begin{equation}\label{e:assum_a}
\log  M(r,a)\le C_0 \log  ^{1+\eta} \frac 1{1-r},\quad r\in [0,1),
\end{equation}
where $ \eta_2>\eta>0$ and $C_0$ is a positive constant.                                                                                                                                                                                                                    The function $\varphi(r)=\exp{\log ^{1+\eta} \frac 1{1-r}}$ has infinite logarithmic order, that is $$\limsup_{r\to1-} \frac{\log  \varphi(r)}{\log  \log  \frac 1{1-r}}=\infty.$$ Using the notation from \cite{CHR2}
\begin{gather*}
  \sigma_\varphi(M_{1/2}(r,a)^{1/2}(1-r)):= \limsup_{r\to1-} \frac{\log  (M_{1/2}(r,a)^{1/2}(1-r))}{\log  \varphi(r)} =\\ = \limsup_{r\to1-} \frac{\log  \Bigl(\frac1{2\pi} \int_{0}^{2\pi} |a(re^{i\theta})|^{\frac 12}\, d\theta  (1-r)\Bigr) }{\log  ^{1+\eta} \frac{1}{1-r}} \le \\
  \le  \limsup_{r\to1-} \frac{\frac 12 \log  M(r,a)+ \log  (1-r)}{\log  ^{1+\eta} \frac{1}{1-r}}\le
  \limsup_{r\to1-} \frac{\frac {C_0}2 \log ^{1+\eta}\frac 1{1-r}+ \log  (1-r)}{\log  ^{1+\eta} \frac{1}{1-r}}=\frac {C_0}2.
\end{gather*}
Then by \cite[Theorem 1]{CHR2} $\limsup_{r\to1-} \frac{\log  T(r,f)}{\log  \varphi(r)}\le \frac {C_0}2.$
Hence
\begin{gather}
\nonumber
  \log  \log  M(r,f) \le \log \Bigl(  T\Bigl( \frac{1+r}{2},f\Bigr)\frac{2}{1-r}\Bigr) \le\\ \le  \Bigl(\frac {C_0}2 +o(1) \Bigr)\log  \varphi(r)\sim \frac {C_0}2 \log ^{1+\eta} \frac{1}{1-r}, \quad r\to1-.\label{e:max_est}
\end{gather}



We write $R_n=1- 2 \cdot 3^{-n-1}$ and $\delta=\frac 14$. Taking into account that $n_{R_n e^{i\theta}}\Bigl(\frac{1-R_n}4\Bigr)=N_{R_n e^{i\theta}}\Bigl(\frac{1-R_n}4\Bigr)=0$, and applying Lemma \ref{l:index} we deduce that \begin{gather}\nonumber
    \Re g(R_ne^{i\theta}) \le \log  M(R_n, f)+|\log  |B(R_n e^{i\theta})|| \le \\ \nonumber \le \log  M(R_n, f) + C\Bigl(\frac 14, 1\Bigr) \sum_{j,k}
 \frac{1-|z_{j,k}|^2}{|1-z_{j,k} R_n e^{i\theta}|} \le \\  \nonumber
 \le e^{ \bigl( \frac {C_0}2+o(1))\bigr) \log  ^{1+\eta} \frac1{1-R_n}} +
 C\Bigl(\frac 14, 1\Bigr) \sum_{|z_{j,k}|\le R_n} 1+ \frac1{1-R_n} \sum_{|z_{j,k}|> R_n} (1-|z_{j,k}|) \le \\ \le\exp \Bigl\{ \Bl\frac {C_0}2+o(1)\Br \log  ^\eta \frac1{1-R_n}\Bigr\} + O\Bigl(\log ^{1+\eta_1}  \frac1{1-R_n}\Bigr) \sim \nonumber \\ \sim  \exp \Bigl\{ \Bigl(\frac {C_0}2+o(1)\Bigr) \log  ^\eta \frac1{1-R_n}\Bigr\}
,\quad  n\to +\infty . \label{e:max_re_g_est}
     \end{gather}
Since $\Re g     $ is harmonic,  $B(r, \Re g)=\max\{ \Re g(re^{i\theta}): \theta \in [0, 2\pi]\}$ is an increasing function.
It follows from \eqref{e:max_re_g_est} and the relation $1-R_n\asymp 1-R_{n+1}$ that
$$ B(r,\Re  g)\le  \exp \Bigl\{ \Bl\frac {C_0}2+o(1)\Br \log  ^{1+\eta} \frac1{1-r}\Bigr\}, \quad r\to1-.$$
The last estimate and Caratheodory's inequality (\cite[Chap.1, \S 6]{Le}) imply
\begin{equation*}
     M(r,g) \le \exp \Bigl\{ \Bl\frac {C_0}2+o(1)\Br \log  ^{1+\eta} \frac1{1-r}\Bigr\}, \quad r\to 1-.
\end{equation*}
Then, applying Cauchy's integral formula, we obtain
\begin{equation*}
     M(r,g') \le \frac{C}{1-r}\exp \Bigl\{ \Bl\frac {C_0}2+o(1)\Br \log  ^{1+\eta} \frac1{1-r}\Bigr\}, \quad r\uparrow1,
\end{equation*}
which yields
\begin{equation}\label{e:upp_est_solut}
     \log M(r,g') \le \Bl\frac {C_0}2+o(1)\Br \log  ^{1+\eta} \frac1{1-r}, \quad r\to 1-.
\end{equation}

We write (cf. \cite{Heit_cmft})
\begin{gather}\nonumber
-g'(z_{n,0})= \frac{B''(z_{n,0})}{2B'(z_{n,0})}=\sum_{k=1}^{m_n-1}\frac{1}{z_{n,0} -z_{n,k}} \frac{1-|z_{n,k}|^2}{1-\bar{z}_{n,k} z_{n,0}}+ \\ +\sum_{j\ne n} \sum_{k=0}^{m_j-1}\frac{1}{z_{j,k} -z_{n,0}} \frac{1-|z_{j,k}|^2}{1-\bar{z}_{j,k} z_{n,0}} =: I_1+I_2.
\end{gather}
It is easy to see that
\begin{gather}\nonumber
  |I_1|\ge \frac12 \frac {m_n}{\varepsilon_n} \sum_{k=1}^{m_{n}-1} \frac 1k\ge \\ \ge  C \exp \Bigl\{ \log  ^{\eta_2+1} \frac1{1-|z_{n,0}|}\Bigr\} \log ^{\eta_1} \frac 1{1-|z_{n,0}|} \log \log  \frac 1{1-|z_{n,0}|}, \quad n\to+\infty.\label{e:i_1_low}
\end{gather}
Then
\begin{gather*}
 |I_2|\le \sum_{j\ne n} \sum_{k=0}^{m_j-1} \frac{1}{|z_{j,k} -z_{n,0}|} \frac{1-|z_{j,k}|^2}{|1-\bar{z}_{j,k} z_{n,0}|}\le \\ \le \sum_{j=1}^{n-1} \sum_{k=0}^{m_j-1} \frac{4}{1-|z_{j,0}|} +\sum_{j=n+1}^\infty \sum_{k=0}^{m_j-1} \frac{ 4(1-|z_{j,k}|)}{(1-|z_{n,0}|)^{2}} \le \\ \le
 4\sum_{j=1}^{n-1} 3^{j+1} (j\log  3)^{\eta_1} +
\frac{4}{(1-|z_{n,0}|)^{2}}  \sum_{j=n+1}^\infty 3^{-j} (j\log  3)^{\eta_1}\le\\ \le  C 3^{n} (n\log  3)^{\eta_1}+ \frac{C}{(1-|z_{n,0}|)^2} \frac{(n\log  3)^{\eta_1}}{3^n} \le \frac{C\log  ^{\eta_1} \frac{1}{1-|z_{n,0}|}}{1-|z_{n,0}|}.
\end{gather*}
Therefore \begin{gather*}|g'(z_{n,0})|=\Bigl| \frac{B''(z_{2n})}{2B'(z_{2n})}\Bigr|\ge \\ \ge C \exp \Bigl\{ \log  ^{\eta_2+1} \frac1{1-|z_{n,0}|}\Bigr\} \log ^{\eta_1}  \frac 1{1-|z_{n,0}|}\log \log  \frac 1{1-|z_{n,0}|}  , \quad n\to+\infty.\end{gather*}
Hence, \begin{equation}\label{e:g'_growth_low_est}
        \log  M(|z_{n,0}|, g')\ge (1+o(1))\log  ^{\eta_2+1} \frac1{1-|z_{n,0}|}, \quad n\to+\infty.
       \end{equation}
This contradicts to \eqref{e:upp_est_solut}. The theorem is proved.
\end{proof}

\begin{proof}[Proof of Theorem \ref{t:zeros_de1}]
  Let $\mathcal{D}_\rho^+(Z)<D<\infty$.  It follows from the definition of $\mathcal{D}_\rho^+(Z)$ that under the assumption of Theorem \ref{t:bdk} that there exist $R_0>0$ and $r_0\in (0,1)$ such that
$n_z(R\rho(z))<D{R^2}$, $R>R_0$ and $|z|\in (r_0,1)$. Separation condition \eqref{e:rho_separ} and the property  \eqref{e:rho_slow_var}  imply that
each disk $U(z, \rho(z)/3)$ contains at most one point $z_k$, and there exists a constant $D_1\ge D$ such that $n_z(R\rho(z))\le D_1 R^2$ for $0<R\le R_0$ and $|z|\in (r_1, 1)$ for some $r_1\in (r_0,1)$. Therefore
\begin{equation}\label{e:n_z_est}
  n_z\Bigl(R\rho(z)\Bigr)<D_1{R^2}, \quad R>0,\; |z|\in (r_1,1).
\end{equation}
The last estimate directly implies
$$n_z\Bl\frac{1-|z|}2\Br \le \frac{D_1}{4} \frac{(1-|z|)^2}{\rho^2(z)}=\frac{D_1}4 {(1-|z|)^2} \Delta h(r)
, \quad |z|=r\in (r_1, 1).$$
We write $\psi(t)=t^{-2} \Delta h(r) 
\Bigr|_{r=1-t^{-1}}$. It follows from the definition of $\sigma(r)$ that $\psi$ is a nondecreasing function on $[1,\infty)$ and  $\psi(2t)=O(\psi(t))$, $t\ge 2$.
Let us estimate $\tilde{\psi}(T)=\int_{1}^{T} \frac{\psi(t)}{t}\,dt$. We have
\begin{gather*}
  \tilde{\psi}(T)=\int_1^T \frac{\Delta h(r)\Bigr|_{r=1-t^{-1}}}{t^3} \, dt=\int_{0}^{1-T^{-1}} (1-r) \Delta h(r) 
  \, dr= \\ =
  \int_{0}^{1-T^{-1}}  (1-r)\frac{(rh'(r))'}r dr= \\ = (1-r) h'(r) \Bigr|_0^{1-T^{-1}}+  \int_{0}^{1-T^{-1}} \frac{h'(r)}r \, dr\le \\ \le \frac{h'(1-T^{-1})}{T}+ 2 h(1-T^{-1}), \quad T\ge 2.
  \end{gather*}
Since $\frac{h'(r)}{h(r)}=O((1-r)^{-1})$, $r\to 1-$, we deduce that $$\tilde\psi\Bigl(\frac{1}{1-r} \Bigr)=O\Bigl(h(r) \Bigr),\quad r\to 1-.$$

We then estimate $N_{z_k} (t)$ for $t\le (1-|z_k|)/2$. Using the separation condition \eqref{e:rho_separ} and  estimate \eqref{e:n_z_est} we deduce
\begin{gather*}
   N_{z_k} ((1-|z_k|)/2)=\int_0^{(1-|z_k|)/2} \frac{(n_\zeta(t)-1)^+}{t}dt \le \int_{\rho(z_k)/2}^{(1-|z_k|)/2} \frac{n_\zeta(t)}{t}dt=  \\ =
  \int_{1/2}^{(1-|z_k|)/(2\rho(z_k))} \frac{n_\zeta(\rho(z_k) \tau )}{\tau}d\tau\le  \int_{1/2}^{(1-|z_k|)/(2\rho(z_k))} D_1 \tau d\tau \le  \\ \le \frac{D_1}{8}\frac{(1-|z_k|)^2}{\rho^2(z_k)}=\frac12 \psi\Bigl( \frac1{1-r}\Bigr).
\end{gather*}



Since $n_z(\frac{1-|z|}2)=O(\psi(\frac{1}{1-|z|}))$, $|z|\in (r_1, 1)$,
 applying Lemma 9 from \cite{ChyShep}  we get the following  estimate of the canonical product
 \begin{equation}\label{e:can_prod_est}
   \log|P(z)|\le C \tilde{\psi}\Bigl(\frac{1}{1-|z|}\Bigr) \le C h(|z|).
\end{equation}

Any  analytic function $f$ in $\D$ with the zero sequence $Z=(z_k)$ can be written in the form  $f(z)=P(z)e^{g(z)}$, where $g$ is  analytic in $\D$. If $f$ is a solution of   \eqref{e:oscil_eq}, then
\begin{equation}
 \label{e:osc_p_g}
P''+2 P'g' +(g'^2 +g'' +a)P=0,
\end{equation}
and, consequently
\begin{equation}
 \label{e:g_interpol}
g'(z_k)=-\frac{P''(z_k)}{2P'(z_k)}=:b_k, \quad k\in \mathbb{N}.
\end{equation}
Therefore, in order to find  a solution of \eqref{e:oscil_eq} with the zero sequence $Z$ we have to find an analytic function $h=g'$ solving the interpolation problem $h(z_k)=b_k$, $k\in \N$.
Using Cauchy's integral theorem  and \eqref{e:can_prod_est} we deduce
\begin{equation*}
 |P''(z_k)|\le \frac{8}{(1-|z_k|)^2} \max _{|z|=\frac{1+|z_k|}{2}} |P(z)|\le \frac{8}{(1-|z_k|)^2} e^{C \tilde \psi (\frac{ 2}{1-|z_k|})}.
\end{equation*}
On the other hand, \eqref{e:p'est} and \eqref{e:error_est_ro} imply  (cf. \eqref{e:ln_prime}) that
$$ \frac{1}{|P'(z_k)|} \le (1-|z_k|) e^{C\tilde \psi (\frac{ 1}{1-|z_k|})}.$$
Hence
\begin{gather*}
 |b_k|=\Bigl| \frac{P''(z_k)}{2P'(z_k)}\Bigr| \le \frac{4}{1-|z_k|} e^{C \tilde \psi (\frac{ 1}{1-|z_k|})}= \\ =
e^{C \tilde \psi (\frac{ 2}{1-|z_k|})+ \log  \frac{4}{1-|z_k|}}\le e^{C \tilde \psi (\frac{ 1}{1-|z_k|})}, \quad k\in \N,
\end{gather*}
because $\tilde\psi (t)/\log  t\to+\infty$ $(t\to+\infty)$. Since the assumptions of Theorem \ref{t:sufficient_cond_interpol} are satisfied, there exists a function $h$  analytic in $\D$ such that $h(z_k)=b_k$ and $\log  M(r,h)\le C \tilde \psi (\frac{1}{1-r})$,
$r\to 1-$, i.e. $\log  M(r,g')\le C \tilde \psi (\frac{1}{1-r})$,
$r\to1-$.

Then, applying Cauchy's integral theorem once more, we get that
$$ M(r,g'')\le \frac{2 }{1-r} M\Bigl( \frac{1+r}{2}, g'\Bigr) \le e^{C \tilde \psi (\frac 1{1-r})}, \quad r \to1-.$$
From \eqref{e:osc_p_g} we obtain
$$ |a(z)| \le \Bigl| \frac{P''(z)}{P(z)}\Bigr|+ 2|g'(z)| \Bigl| \frac{P'(z)}{P(z)}\Bigr|+ |g'(z)|^2+ |g''(z)|.$$
It follows from results of \cite{ChyGuHe} (or \cite{CHR1}) that for any $\delta>0$ there exists a set $E_
\delta \subset [0,1)$ such that
\begin{equation*}
 \max \Bigl\{  \frac{|P''(z)|}{|P(z)|}, \frac{|P'(z)|}{|P(z)|} \Bigr\} \le \frac{1}{(1-|z|)^q}, \quad |z|\in [0,1)\setminus E_\delta,
\end{equation*}
where $q\in (0,+\infty)$,  and  $m_1(E_\delta \cap [r,1))\le \delta(1-r)$ as $r \uparrow1$.
Thus, \begin{equation}\label{e:a_est_excep}
       |a(z)| \le e^{\tilde C\tilde\psi (\frac{1}{1-|z|})} , \quad |z|\in [0,1)\setminus E.
      \end{equation}
Since $M(r,a)$ increases,  condition \eqref{e:polya} and Lemma 4.1 from \cite{CHR1} imply that  inequality \eqref{e:a_est_excep} holds for all $z\in \D$ for an appropriate choice of $\tilde C$.
\end{proof}

\begin{rmk} Though the hypotheses of Theorem \ref{t:zeros_de1} are similar to those of Theorem \ref{t:bdk} we do not use the interpolating function constructed in \cite{Borich_k}. We need  an interpolating function $f$ to have the property $f(z)\asymp \mathop{\rm dist}(z, Z_f) p(z)$, where $ Z_f$ is the zero set of $F$, $p(z)$ some nonvanishing continuous function in $\D$. But it seems that it is not the case.
\end{rmk}

{\it Address 1:} Faculty of Mechanics and Mathematics, Lviv Ivan Franko National University,
Universytets'ka 1, 79000,
 Lviv,  Ukraine

{\it e-mail:} chyzhykov@yahoo.com

{\it Address 2:} School of Mathematics Science,
           Guizhou  Normal University, Guiyang, Guizhou 550001, China

{\it e-mail:} longjianren2004@163.com


\begin{thebibliography}{19}\parskip=-1pt
%

\bibitem{BerTay} Berenstein C.A., Taylor B.A. A new look at interpolation theory
for entire functions of one variable, Advances in Mathematics,  {\bf 33}, 109--143  (1979)

\bibitem{Borich_k} Borichev A., Dhuez R., Kellay K. Sampling and interpolation in large Bergman and Fock space, J. Funct. Analysis {\bf  242}, 563--606  (2007)
%

 \bibitem{ChyShep}  Chyzhykov I.,  Sheparovych I., { Interpolation of analytic functions of moderate growth in the unit
disc and zeros of solutions of a linear differential equation,} J. Math. Anal. Appl. {\bf 414},   319-333 (2014)

\bibitem{CHR2}  Chyzhykov I., Heittokangas J., R\"atty\"a J., { On the finiteness of $\varphi$-order of solutions of linear differential  equations in the unit disc},
J. d'Analyse Math. {\bf 109} (1), 163--196  (2010)
%
%
\bibitem{ChyGuHe}  Chyzhykov I., Gundersen G. , Heittokangas J. , Linear differential equations and logarithmic derivative estimates, Proc. London Math. Soc. {\bf 86} (3), 735--754  (2003)
\bibitem{CHR1} Chyzhykov I., J.~Heittokangas, J.~R\"atty\"a, {Sharp logarithmic derivative estimates with applications to ODE's in the unit disc,} J. Australian Math. Soc.  {\bf 88}, 145--167  (2010)
%
%
%
\bibitem{DrSh} Drasin D. , Shea D., P\'olya peaks and the oscillation of positive functions, Proc. Amer. Math. Soc. {\bf  34}, 403--411  (1972)
%
%
\bibitem{Gron19}  Gr\"ohn J., Solutions of complex differential equation having pre-given zeros in the unit disc, Constr. Approx. {\bf 49}, 295--306  (2019)
%

\bibitem{Gr_Heit}  Gr\"ohn J., Heittokangas J., New findings on Bank-Sauer approach in oscillatory theory, Constr. Approx. {\bf  35}, 345--361  (2012)

\bibitem{GrNikRat} Gr\"ohn J.,  Nikolau A., R\"atty\"a J., Mean growth and geometric zero distribution of solutions of linear differential equations, J. Anal. Math. {\bf 134}, 747--768  (2018)
%

\bibitem{HarMas}  Hartmann A., Massaneda X., Interpolating sequences for holomorphic functions of restricted growth,
Ill. J. Math. {\bf 46} (3) 929--945  (2002)

\bibitem{Heit_cmft} Heittokangas J., Solutions of $f''+A(z)f=0$ in the unit disc having Blaschke sequence as zeros, Comp. Meth. Funct. Theory {\bf  5} (1), 49--63  (2005)
%
\bibitem{Heitt_sur}  Heittokangas J.,  A survey on Blaschke-oscillatory differential equations, with updates,  in  Blaschke products and their applications, Fields Institute Communicatios, Vol. 65, J.Mashreghi, E.Fricain (eds.),   43--98 (2012)
%
%

\bibitem{Le}  Levin B. Ja. {\it Distribution of zeros of entire functions},
revised edition, Transl. Math. Monographs, Volume 5, translated by
R. P. Boas {\it et al}, Amer. Math. Soc., Providence (1980)



%
%
\bibitem{Seda}  \v Seda V., On some properties of solutions of the differential equation $y''=Q(z)y$, where $Q(z)\ne 0$ is an entire function, Acta. Fac. Nat. Univ. Comenian math. {\bf 4}, 223--253  (1959) (in Slovak)
%
\bibitem{Seip93} K. Seip, Beurling type density theorems in the unit disc, Invent. math. {\bf 113}, 21--39  (1993)
%
%
%
%
%
%
%
%
%
%
%
%
%
%
%
%
%
%
\bibitem{Vyn_Shav12}    Vynnyts'kyi B.,  Shavala O., Remarks on \v Seda theorem, Acta. Math. Univ. Comenianae, {\bf  LXXXI} (1), 55--60 (2012).


%
 \end{thebibliography}
\end{document}